\documentclass[reqno,12pt,a4paper]{amsart}
\usepackage{graphicx} 
\usepackage{amsmath, amsthm, amsfonts, amssymb, amscd, bm, mathrsfs}
\usepackage{verbatim}
\usepackage{enumerate}
\usepackage[numbers]{natbib}
\usepackage{array}
\usepackage{color}
\usepackage[colorlinks=true,linkcolor=blue,linktoc=page]{hyperref}

\newtheorem{thm}{Theorem}[section]
\newtheorem{cor}[thm]{Corollary}
\newtheorem{lemma}[thm]{Lemma}
\newtheorem{conj}[thm]{Conjecture}
\theoremstyle{definition}
\newtheorem*{Def}{Definition}

\theoremstyle{remark}

\DeclareMathOperator{\Per}{per_1}
\DeclareMathOperator{\sPer}{per_s}

\DeclareMathOperator{\NPer}{per}
\DeclareMathOperator{\Hull}{Hull}
\DeclareMathOperator{\Zer}{Zero}

\DeclareMathOperator{\Supp}{supp}

\DeclareMathOperator{\CongN}{\equiv_{\it n}\,}

\newcommand{\floor}[1]{\left\lfloor #1 \right\rfloor}
\newcommand{\sym}{\mathcal S}
\newcommand{\corp}{\mathcal P}
\newcommand{\corh}{\mathcal H}

\renewcommand{\geq}{\geqslant}

\renewcommand{\ge}{\geqslant}
\renewcommand{\le}{\leqslant}

\def\sref#1{\S$\ref{#1}$}
\def\lref#1{Lemma~$\ref{#1}$}
\def\tref#1{Theorem~$\ref{#1}$}

\def\cyref#1{Corollary~$\ref{#1}$}
\def\cjref#1{Conjecture~$\ref{#1}$}

\begin{document}

\title{Multidimensional permanents of polystochastic matrices}

\author{Billy Child and Ian M. Wanless}

\thanks{
School of Mathematics, Monash University, Vic 3800, Australia.
{\tt william.child@monash.edu, ian.wanless@monash.edu}.
Research supported by Australian Research Council grant DP150100506.
}

\maketitle

\begin{abstract}
A $d$-dimensional matrix is called \emph{$1$-polystochastic} if it is
non-negative and the sum over each line equals~$1$. Such a matrix that
has a single $1$ in each line and zeros elsewhere is called a
\emph{$1$-permutation} matrix. A \emph{diagonal} of a $d$-dimensional
matrix of order $n$ is a choice of $n$ elements, no two in the same
hyperplane. The \emph{permanent} of a $d$-dimensional matrix is the
sum over the diagonals of the product of the elements within the
diagonal.

For a given order $n$ and dimension $d$, the set of $1$-polystochastic
matrices forms a convex polytope that includes the $1$-permutation
matrices within its set of vertices. For even $n$ and odd $d$, we give
a construction for a class of $1$-permutation matrices with zero
permanent. Consequently, we show that the set of $1$-polystochastic
matrices with zero permanent contains at least $n^{n^{3/2}(1/2-o(1))}$
$1$-permutation matrices and contains a polytope of dimension at least
$cn^{3/2}$ for fixed $c,d$ and even $n\to\infty$.  We also provide
counterexamples to a conjecture by Taranenko \cite{TaranenkoBig} about
the location of local extrema of the permanent.

For odd $d$, we give a construction of $1$-permutation matrices that
decompose into a convex linear sum of positive diagonals. These combine
with a theorem of Taranenko \cite{TaranenkoBig} to provide counterexamples
to a conjecture by Dow and Gibson \cite{GibsonDow} generalising
van der Waerden's conjecture to higher dimensions.


\end{abstract}

\section{Introduction}

Let $A=[A(i,j)]$ be a square matrix of order $n$. The permanent of $A$ is the unsigned determinant: 
\[\NPer(A)=\sum_{\sigma\in\sym_n}\prod_{i=1}^nA(i,\sigma(i)),\]
where the sum is over all permutations in the symmetric group $\sym_n$. 

A non-negative matrix $A$ is \emph{doubly stochastic} if the sum along every row and column of $A$ equals $1$. One of the most celebrated results in the history of permanents is this:

\begin{thm}\label{t:vdW}
Amongst the doubly stochastic matrices of order $n$, the minimum value of the permanent is $n!n^{-n}$ and it is obtained uniquely at $n^{-1}J_n$, where $J_n$ denotes the square matrix of order $n$ with a $1$ in every entry.
\end{thm}

\tref{t:vdW} was first conjectured by van der Waerden in 1926, and was a famous open problem inspiring the development of much of the theory on permanents up until 1981, when two proofs appeared independently in the same year by Egorychev \cite{Egorychev} and Falikman \cite{Falikman}. 
Important to both proofs is that the permanent of a doubly stochastic matrix is positive and the resulting fundamental result proven by Birkhoff in 1946 \cite{BirkhoffOriginal}.

\begin{thm}\label{t:Birk}
If $A$ is a doubly stochastic matrix, then there are permutation matrices $P_1,\dots,P_s$ and positive constants $c_1,\dots,c_s$ such that \[A=\sum_{i=1}^sc_iP_i\] and $\sum_ic_i=1$.
\end{thm}

In this paper we examine attempts to generalise Birkhoff's Theorem and van der Waerden's conjecture to higher dimensional matrices.
For a positive integer $n$, let $I_n=\{1,\dots,n\}$. For $d\ge2$, a \emph{$d$-dimensional matrix $A$ of order $n$} is an array $I_n\times\cdots\times I_n\to \mathbb R$. For $(i_1,\dots,i_d)\in I_n\times\cdots\times I_n$, we refer to $A(i_1,\dots,i_d)$ as the $(i_1,\dots,i_d)$ \emph{element} of $A$. Denote by $M(d,n)$ the set of $d$-dimensional matrices of order $n$. 

A \emph{submatrix} of a matrix $A$ is a restriction of $A$ to $L_1\times\cdots\times L_d$, where $L_i\subseteq I_n$ for each $i$. If for all $i$ we have $|L_i|\in\{1,n\}$, then the submatrix is called a $k$-\emph{plane}, where $k$ is the number of subsets $L_i$ with $|L_i|=n$. We will call a $(d-1)$-plane of $A$ a \emph{hyperplane}, and a $1$-plane of $A$ a \emph{line}.

For $A\in M(d,n)$ and $1\le s\le d$, define the \emph{$s$-permanent} of $A$, denoted $\sPer(A)$, to be the sum of all products of $n^s$ elements of $A$, no two in the same $(d-s)$-plane. We will call such a selection of $n^s$ elements, no two in the same $(d-s)$-plane, an $s$-\emph{diagonal} of $A$, and say an $s$-diagonal is \emph{positive} if all the elements of the diagonal are positive. If no $s$-diagonal exists then the $s$-permanent is an empty sum and hence equal to zero. Throughout this paper we are primarily interested in the $1$-permanent and $1$-diagonals. However, we acknowledge that all of the questions that we investigate have analogues for $s$-permanents that would be worth pursuing. Unless stated otherwise, by a \emph{diagonal} we will mean a $1$-diagonal and by \emph{permanent} we will mean the $1$-permanent.

Let $\Lambda_s(d,n)$ be the set of $(0,1)$-matrices in $M(d,n)$ with precisely one $1$ in each $s$-plane. A matrix in $\Lambda_s(d,n)$ will be called an $s$-\emph{permutation matrix}. Note that by placing ones on a diagonal and zeros elsewhere we get a $(d-1)$-permutation matrix. A non-negative $A\in M(d,n)$ is called \emph{$s$-polystochastic} if the sum of the elements within every $s$-plane is equal to 1. Denote by $\Omega_s(d,n)$ the set of $s$-polystochastic $d$-dimensional matrices of order $n$. The set $\Omega_s(d,n)$ is bounded and defined by finitely many linear inequalities, hence it forms a convex polytope and is given by the convex hull of its vertex set.

A \emph{Latin hypercube} of dimension $d$ and order $n$ is a $d$-dimensional matrix of order $n$ with the property that every line contains precisely one of each element in $I_n$. For each $d$-dimensional $1$-permutation matrix $P$ of order $n$ there is a corresponding Latin hypercube $H$ of dimension $d-1$ and order $n$, and vice versa. The correspondence is that $P(i_1,\dots,i_d)=1$ if and only if $H(i_1,\dots,i_{d-1})=i_d$. We will write $P=\corp(H)$ and $H=\corh(P)=\corp^{-1}(P)$. Furthermore, we see that a positive diagonal in $P$ corresponds to a \emph{transversal} in $H$, that is, a selection of $n$ elements, no two in the same hyperplane or sharing the same symbol. Hence $\Per(P)$ counts transversals in $H$. On several occasions we will also use a concept we call a \emph{mixed transversal} of a set $\{H_1,\dots,H_k\}$ of Latin hypercubes of order $n$. By this we mean a selection of $n$ elements $\{\alpha_1,\dots,\alpha_n\}$ with each $\alpha_i\in H_j$ for some $j$ that may depend on $i$, and such that any two elements differ in every coordinate and contain different symbols.

There is a natural action of the wreath product $\sym_n\wr\sym_d$ on $M(d,n)$, where $\sym_d$ permutes the coordinates and each copy of $\sym_n$ permutes the values within one coordinate. The same action works on $\Lambda_s(d,n)$ and $\Omega_s(d,n)$. Given the above correspondence with $\Lambda_1(d,n)$, this then induces an action of $\sym_n\wr\sym_d$ on Latin hypercubes of order $n$ and dimension $d-1$. A \emph{species} of Latin hypercubes or polystochastic matrices is an orbit under these actions.

\tref{t:Birk} can be restated as saying that the set of doubly stochastic matrices is the \emph{convex hull} of the permutation matrices. By the convex hull we mean the set of convex combinations, that is $x\in\Hull(X)$ if $x$ can be expressed as a combination $x=\sum_ia_ix_i$ for some $x_i\in X$ and coefficients $a_i>0$ with $\sum_i a_i=1$.  It is a straightforward exercise to see that the proof of Birkhoff's theorem in dimension $2$ generalises to give a decomposition of matrices in $\Omega_s(d,n)$ into $s$-permutation matrices, provided the $(d-s)$-permanent is positive on $\Omega_s(d,n)$. Jurkat and Ryser \cite{JurkatRyser} pose the problem of decomposing $s$-polystochastic matrices into $s$-permutation matrices for $s\geq 2$ as a generalisation of Birkhoff's theorem. 

Dow and Gibson \cite{GibsonDow} were the first to seriously consider permanents of higher dimensional matrices. They showed that there are $(d-1)$-polystochastic matrices for every dimension $d\ge3$ and order $n\ge2$ (with $n\ne3$) which have vanishing permanent (their paper does not exclude the $n=3$ case, but their construction fails in that case). Realising this meant that there could be no higher dimensional analogue of Birkhoff's theorem for $(d-1)$-polystochastic matrices, they instead made the following conjecture about the permanent on the convex hull of the $(d-1)$-permutation matrices.

\begin{conj}\label{DGconj} 
If $A\in\Hull(\Lambda_{d-1}(d,n))$, then $\Per(A)\ge(n!/n^n)^{d-1}$, with equality if and only if $A=n^{1-d}J_n^d$.
\end{conj}

Here $J_n^d$ is the $d$-dimensional matrix of order $n$ with all elements equal to $1$. They showed that the conjecture holds for order $2$.

Another approach to generalising van der Waerden's conjecture would be
to look at the $1$-polystochastic matrices in dimensions and orders where the minimum
permanent is not known to be zero, possibly leading to:

\begin{conj}\label{DGconjmod} 
If $A\in\Omega_1(d,n)$ where $d$ is even or $n$ is odd, then
$\Per(A)\ge(n!)^{d-1}/n^n$, with equality if and only if
$A=n^{-1}J_n^d$.
\end{conj}

A counterexample to \cjref{DGconj} might not be 1-polystochastic and
hence might not yield a counterexample to
\cjref{DGconjmod}. Conversely, a counterexample to \cjref{DGconjmod}
might not be decomposable into a sum of positive diagonals and hence
might not yield a counterexample to \cjref{DGconj}. Therefore it is not clear
whether either of these conjectures implies the other for any
given values of $d$ and $n$.

Taranenko \cite{TaranenkoLS} proved the following result, which shows
that \cjref{DGconjmod} fails for odd dimensions.

\begin{thm}\label{JMinMax}
  The matrix $n^{-1}J_n^d$ is a local extrema of the permanent amongst $\Omega_1(d,n)$. If $d$ is even, $n^{-1}J_n^d$ is a local minimum, and if $d$ is odd, $n^{-1}J_n^d$ is a local maximum.
\end{thm} 

In a subsequent paper \cite{TaranenkoBig}
Taranenko offered this conjecture about where other local extrema can occur.

\begin{conj}\label{cj:centreface}
  All local extrema of the permanent on the polytope $\Omega_1(d,n)$
  are located at the vertices or centres of its faces. 
\end{conj}

In the same paper she showed that \cjref{DGconjmod} fails asymptotically:

\begin{thm}
Suppose that the order $n$ is odd or the dimension $d\ge 3$ is even. Then there exists a $1$-polystochastic $d$-dimensional matrix $A$ of order $n$ whose permanent is asymptotically less than that of the uniform matrix:
\[\Per(A)\le\left(\frac{1}{2}+o(1)\right)^n\Per(n^{-1}J_n^d) \text{\; as\; $n\rightarrow\infty$.}\]
\end{thm}

She also showed that \cjref{DGconjmod} fails for $n=3$ and $d\ge3$.
In \sref{s:DG}, we use \tref{JMinMax} to show that \cjref{DGconj}
fails for odd dimensions.

As mentioned already, Dow and Gibson found $(d-1)$-polystochastic matrices with vanishing permanent. There are also known examples of $1$-polystochastic matrices that have vanishing permanent, which we generalise in \tref{ZeroThm}.  As a corollary we provide counterexamples to \cjref{cj:centreface}. However, for even dimension or odd order, there are no known examples of $1$-polystochastic matrices with vanishing permanent, and Taranenko \cite{TaranenkoBig} has conjectured that the permanent is always positive.

\begin{conj} \label{ZeroConj} 
All $1$-polystochastic matrices of even dimension or odd order have positive permanent. 
\end{conj}

In support of this conjecture, Taranenko proved the order $3$ case. More recently Taranenko \cite{Taranenko4} showed that Conjecture \ref{ZeroConj} is true in dimension $4$ and order $4$. Since dimension $2$ is the classic result that doubly stochastic permutation matrices have positive permanent, this was one of the smallest of the unresolved cases. Taranenko \cite{TaranenkoBig} surveys some conditions on the non-zero elements of a matrix to guarantee the positivity of the permanent.

While there are known examples of $1$-polystochastic matrices of odd order that cannot be decomposed into a sum of diagonals, there are no such known examples for even dimensional $1$-polystochastic matrices. Taranenko \cite{TaranenkoBig} has conjectured that a decomposition is always possible in even dimensions. 

\begin{conj} All $1$-polystochastic matrices of even dimension can be represented as a non-negative linear combination of $(d-1)$-permutation matrices.
\end{conj}

While $1$-polystochastic matrices may not decompose into diagonals, it is still true that $\Omega_1(d,n)$ is the convex hull of its vertex set. The $1$-permutation matrices are vertices of this polytope, but in general there are others. There is a simple characterisation of the vertices in terms of their supports first given by Jurkat and Ryser \cite{JurkatRyser}, though stated in very different terminology than we use here.

\begin{thm} 
A matrix $A\in\Omega_1(d,n)$ is a vertex if and only if $A$ has minimal support, that is, if $\Supp(A)\supseteq\Supp(B)$ for $B\in\Omega_1(d,n)$, then $A=B$.
\end{thm}

Ke, Li and Xiao \cite{KeLiXiao} give a formulation of this condition into a system of linear equations in the elements of $A$ for which a unique solution exists if and only if $A$ is a vertex. Using this formulation, they compute the entire set of vertices for $\Omega_1(3,3)$ and $\Omega_1(3,4)$. Finally, Linial and Luria \cite{LinLur3dim} give a lower bound on the number of vertices of $\Omega_1(3,n)$ as $n\to\infty$.

The structure of this paper is as follows. In \sref{s:zero} we
consider the zeros of the permanent on $s$-polystochastic
matrices.  We show that whenever the order is even and the dimension
is odd there is a large set of vertices such that the permanent of any
linear combination of these vertices is zero.  As a consequence we
deduce that \cjref{cj:centreface} fails.  In \sref{s:DG} we show that
\cjref{DGconj} fails for odd dimensions.

\section{Properties of the zero set}\label{s:zero}

In this section we consider
$\Zer_s(d,n)=\{A\in\Omega_s(d,n):\Per(A)=0\}$, the zero set of the
permanent amongst $s$-polystochastic matrices of order $n$ and
dimension $d$. Since $\Zer_s(d,n)$ is the zero set of a multivariate
polynomial on $M(d,n)$ and is contained inside the bounded set $\Omega_s(d,n)$, it is compact. We can also make the following observation about its structure.

\begin{lemma}\label{l:unionofpolytopes}
Let $s,d,$ and $n$ be positive integers, then $\Zer_s(d,n)$ is the
union of finitely many polytopes.
\end{lemma}

\begin{proof}
Whether or not the permanent is zero on a nonnegative matrix
depends only on the support. For any set $V$ of vertices of
$\Omega_s(d,n)$ there is a polytope $P_V=\Hull(V)$.  We claim that
either $P_V\subseteq\Zer_s(d,n)$ or the interior of $P_V$ is disjoint
from $\Zer_s(d,n)$. This is because $\Zer_s(d,n)$ is compact and all
points in the interior of $P_V$ have the same support. The boundary of
a polytope is itself a union of polytopes.  There are only finitely
many options for $V$, and $\Zer_s(d,n)$ is the union of $P_V$ over all
choices of $V$ for which $P_V\subseteq\Zer_s(d,n)$.
\end{proof}

It is important to note that the polytopes in
\lref{l:unionofpolytopes} are not necessarily disjoint. A concrete
example will be given after \tref{ZeroThm}.  

We say that the \emph{dimension} of $\Zer_s(d,n)$ is the maximum of
the dimensions of the polytopes $P_V$ contained in $\Zer_s(d,n)$.  One
of our goals for this section is to find a lower bound on the
dimension of $\Zer_1(d,n)$ for odd $d$ and even $n$ (when we know
that the permanent can be zero).

We also note that the permanent is a multi-linear function. It is
immediate that $\Omega_s(d,n)\setminus \Zer_s(d,n)$ is convex, since
all positive linear combinations of matrices with positive permanent
have positive permanent. A convex polytope with vertices
$P_1,\dots,P_m$ is necessarily closed, having elements
$a_1P_1+\cdots+a_mP_m$ given by non-negative solutions to the linear
equation $a_1+\cdots+a_m=1$. It follows that, if $\Zer_s(d,n)$ is
non-empty, then $\Omega_s(d,n)\setminus \Zer_s(d,n)$ is not a convex
polytope, given that $\Zer_s(d,n)$ is closed and $\Omega_s(d,n)$ is
convex and hence connected. Hence, $\Omega_s(d,n)\setminus
\Zer_s(d,n)$ is convex but is not a polytope when $\Zer_s(d,n)$ is
non-empty.

As mentioned in the introduction, in the even $n$ and odd $d$ case, it is well known that there are $1$-permutation matrices with zero permanent. We now give a large family of such matrices inspired by the result in \cite{IanNickTransversals} on Latin squares with no transversals.

For the following, let $\CongN$ denote congruence mod $n$.

\begin{Def}\label{LinDef} 
A hypercube $H\in M(d,n)$ is \emph{linear} if there exists $s\in \{1,\dots,n\}$ and $c_i\in\{1,\dots,n\}$ for $i=1,\dots,d$ such that
\begin{equation}\label{e:linear}
H(x_1,\dots,x_d)\CongN s+\sum_{i=1}^dc_ix_i.
\end{equation}
This hypercube will be Latin if and only if all $c_i$ are relatively
prime to $n$. A permutation matrix $P\in\Lambda_1(d+1,n)$ is
\emph{linear} if $\corh(P)$ is linear.  Define $C_{d,n}\in M(d,n)$ to
be the cyclic Latin hypercube of dimension $d$ and order $n$, obtained
by \eqref{e:linear} with $s\CongN 0$ and linearity coefficients
$c_i=1$ for all $i$.
\end{Def}

\begin{Def} For a matrix $A\in M(d,n)$, and an element $e=A(i_1,\dots, i_d)$, let the \emph{Delta function} be given by $\Delta(e)\CongN e-i_1-\cdots-i_d$. 
\end{Def}

The Delta function can be viewed as measuring the difference between a
matrix in $M(d,n)$ and $C_{d,n}$; the Delta function is zero
wherever a matrix agrees with $C_{d,n}$, and counts the difference
(mod~${n}$) wherever a matrix differs from $C_{d,n}$.

Variants of the following Lemma have been used to solve a wide variety of problems involving transversals and their generalisations \cite{transurv}.

\begin{lemma}[Delta lemma]\label{l:delta} 
Let $P\in\Lambda_1(d,n)$ and $T=\{\alpha_1,\dots,\alpha_n\}$ be a transversal of $\corh(P)$. Then, 
\[\sum_{i=1}^n \Delta(\alpha_i)\CongN\begin{cases} 0&\quad\text{if } n \text{ is odd or } d \text{ is even,}\\ {n}/{2}&\quad \text{otherwise}.\end{cases}\]
\end{lemma}

\begin{proof} We have,
\[\sum_{i=1}^n \Delta(\alpha_i)\CongN \sum_{i=1}^{n}i-(d-1)\sum_{i=1}^{n}i\CongN\frac{1}{2}(2-d)n(n+1),\] 
which is an even multiple of $n/2$ if $n$ is odd or $d$ is even, and an odd multiple otherwise.
\end{proof}

It was shown in \cite{IanNickTransversals} that for $r(r-1)<n$, any Latin square that agrees with the cyclic Latin square outside $r$ consecutive rows has no transversals. Making only slight adjustments to accommodate higher dimensions, we get:

\begin{thm}\label{ZeroThm}
Let $n$ be even and $d$ odd, and $r$ a positive integer such that $r(r-1)< n$. Then any $P\in\Lambda_1(d,n)$ for which $\corh(P)$ agrees with $C_{d-1,n}$ in all but $r$ consecutive hyperplanes has zero permanent. Furthermore, if we fix the $r$ hyperplanes on which the hypercubes can differ from $C_{d-1,n}$, then all linear combinations of the corresponding permutation matrices also have zero permanent.
\end{thm}

\begin{proof} Suppose that $T=(\alpha_1,\dots,\alpha_n)$ is a transversal in $\corh(P)$ with elements chosen from hyperplanes $P_1,\dots,P_n$ respectively. We assume that $\corh(P)$ agrees with $C_{d-1,n}$ on $P_i$ for $1\le i\le n-r$. We know then that $\Delta(\alpha_i)\CongN 0$ for $1\le i\le n-r$. For $n-r+1\le i\le n$, the element $\alpha_i$ is contained in a single line $\ell_i$ that intersects $P_1,\dots, P_n$. Let $\sigma_i$ be the symbol of the element in $\ell_i\cap P_1$. Then the symbols $\sigma_i,\dots,\sigma_i+n-r-1 \pmod{n}$ appear at the intersections of $\ell_i$ and $P_1,\dots,P_{n-r}$ and the remaining symbols in $\ell_i$ are $\sigma_i+n-r,\dots, \sigma_i-1 \pmod{n}$. Also, $C_{d-1,n}$ has symbol $\sigma_i+i-1$ at the coordinates of $\alpha_i$. Define $k_i$ by $\Delta(\alpha_i)\CongN k_i\in(-\frac{n}{2},\frac{n}{2}]$. Then it follows from the above and $r\le n/2<i\le n$, that $n-r-i+1\le k_i\le n-i$. Hence, defining $K$ by $\sum_{i=1}^n \Delta(\alpha_i)\CongN K\in (-\frac{n}{2},\frac{n}{2}]$, we have
\[|K|\le\left|\sum_{i=n-r+1}^n k_i\right| \le 
\max\left(\sum_{i=n-r+1}^n|n-r-i+1|,\sum_{i=n-r+1}^n|n-i|\right)=
\frac{1}{2} r(r-1)<\frac{n}{2}.\]
But by \lref{l:delta}, $K=n/2$, so no such transversal $T$ exists.

Now suppose $A_1,\dots, A_m$ are permutation matrices with corresponding hypercubes $H_1,\dots, H_m$ that agree with $C_{d-1,n}$ on hyperplanes $P_1,\dots,P_{n-r}$ and (possibly) differ on consecutive hyperplanes $P_{n-r+1},\dots,P_n$. Consider the support of the sum $A=\Supp(A_1+\cdots+ A_m)$. The existence of a positive diagonal of $A$ is equivalent to the existence of a mixed transversal of $\{H_1,\dots,H_m\}$. Suppose $T=(\alpha_1,\dots,\alpha_n)$ is such a mixed transversal. Then $\Delta(\alpha_i)\CongN 0$ for $1\le i\le n-r$, and for $n-r+1\le i\le n$, defining $k_i$ by $\Delta(a_i)\CongN k_i\in(-\frac{n}{2},\frac{n}{2}]$, we again have $n-r-i+1\le k_i\le n-i$. So, again, no such mixed transversal can exist.
\end{proof}

As an example of \lref{l:unionofpolytopes} and \tref{ZeroThm}, let us
consider $\Zer_1(3,4)$.  There are 12 distinct hypercubes that can be
obtained from interchanging consecutive hyperplanes of $C_{2,4}$
(indexing the hyperplanes in any one direction modulo $4$, so that the
first and last planes are considered consecutive).  \tref{ZeroThm}
shows that each of these 12 hypercubes results in a point in
$\Zer(3,4)$ that is joined to $\corp(C_{2,4})$ by a line within
$\Zer(3,4)$.  An easy computation confirms that these 12 lines are
maximal polytopes within $\Zer(3,4)$ and that no other polytope
includes $\corp(C_{2,4})$.  For this, we use the catalogue of vertices
of $\Omega_1(3,4)$ that was identified in \cite{KeLiXiao}.  The only
vertices that are zeros of the permanent are the 432 vertices that are
equivalent to $\corp(C_{2,4})$, modulo permutations within each
coordinate. Hence each vertex of $\Zer_1(3,4)$ is connected by 12
lines to other vertices and is not in any higher dimensional
polytopes. In particular, $\Zer_1(3,4)$ is 1-dimensional, and it is
also easy to see from the above that it is connected.

We have just seen that $\Zer_1(3,4)$ contains a single species of
vertices. However, \tref{ZeroThm} implies that, as $n$ grows, the
number of vertices in $\Zer_1(d,n)$ grows rapidly.

\begin{cor}\label{ZeroVertices} 
For even $n$ and odd $d$, $\Zer_1(d,n)$ contains at least $n^{n^{3/2}(1/2-o(1))}$ species of vertices of $\Omega_1(d,n)$, for fixed $d$ as $n\to \infty$.
\end{cor} 

\begin{proof} 
From \cite{IanNickTransversals} we know that the number of species of Latin squares that agree with the cyclic Latin square of order $n$ on the first $n-r$ rows is $n^{n^{3/2}(1/2-o(1))}$, where $r$ is the largest integer satisfying $r(r-1)<n$ as $n\to \infty$. For each such Latin square $L$, we can construct a Latin hypercube of dimension $d-1$ by $H(x_1,\dots,x_{d-1})\CongN L(x_1,x_2)+\sum_{i=3}^{d-1}x_i$. The resulting Latin hypercube $H$ agrees with $C_{d-1,n}$ on all but $r$ consecutive hyperplanes. The number of Latin hypercubes of dimension $d-1$ and order $n$ in a species is at most $|\sym_n\wr\sym_d|=d!(n!)^d=n^{O(n)}$ for fixed $d$, giving the desired bound. 
\end{proof}

A \emph{Latin subrectangle} of a Latin square $L$ is a submatrix of $L$ in which each row is a permutation of the same set of symbols. Let a \emph{$k$-cycle} (sometimes a \emph{$k$-row-cycle}) in a Latin square be a $2\times k$ Latin subrectangle that contains no $2\times k'$ Latin subrectangle for $k'< k$. A $2$-cycle is called an \emph{intercalate}. By \emph{switching a cycle} in a Latin square we will mean altering the cycle by interchanging its rows, thereby creating a different Latin square.

We will say that two Latin hypercubes are \emph{linearly independent} if their corresponding permutation matrices are linearly independent. Our next result gives a lower bound on the dimension of $\Zer_1(d,n)$. For comparison, the dimension of $\Omega_1(d,n)$ is $(n-1)^d$.

\begin{cor}\label{ZeroDim}
  There exists a constant $c$ such that the dimension of $\Zer_1(d,n)$ is
  at least $cn^{3/2}$ for all even $n>2$ and odd $d>1$.
\end{cor}

\begin{proof} 
First suppose that $n\in\{4,6\}$. Let $D$ be the Latin hypercube
formed from $C_{d-1,n}$ by interchanging two consecutive
hyperplanes. By \tref{ZeroThm}, any linear combination of $\corp(D)$
and $\corp(C_{d-1,n})$ has zero permanent. Thus $\Zer_1(d,n)$ has
dimension at least $1$.

Thus we may assume for the remainder of the proof that $n>6$.  Let $r$
be as large as possible, subject to $r(r-1)<n$. We bound from below
the size of the largest linearly independent set of Latin squares that
differ from the cyclic Latin square $C_{2,n}$ only within the first
$r$ rows.

For integers $a$ and $b$ we define a new Latin square $L_{a,b}$ that
differs from $C_{2,n}$ by switching two cycles. Start by switching the
$\frac{n}{2}$-cycle in $C_{2,n}$ between rows $2a-1$ and $2a+1$ that
contains $C_{2,n}(2a-1,2b-1)$, and then switch the newly created
intercalate containing $C_{2,n}(2a,2b-1)$ and $C_{2,n}(2a,2b)$. This allows us
to swap the symbols in cells $(2a,2b-1)$ and $(2a,2b)$ while leaving
the rest of row $2a$ untouched, and otherwise only changing rows
$2a-1$ and $2a+1$. A given $L_{a,b}\in S=\{L_{a,b}:1\le
a\le\floor{\frac{r-1}{2}},1\le b\le\frac{n}{2}\}$ is the only square
in $S$ to differ from $C_{2,n}$ in cell $(2a,2b)$. Hence, $S$ is a linearly
independent set of $\frac{n}{2}\floor{\frac{r-1}{2}}$ Latin squares
that agree with $C_{2,n}$ on all but the first $r$ rows.

We then fill out the Latin squares in $S$ into $d$-dimensional Latin
hypercubes as described in Corollary \ref{ZeroVertices}. These are
clearly still linearly independent and, by Theorem \ref{ZeroThm}, all
linear combinations of the corresponding permutation matrices have
zero permanent. Hence, the zero set must have dimension at least
$\frac{n}{2}\floor{\frac{r-1}{2}}$. The result follows. 
\end{proof}

Note that \cyref{ZeroDim} cannot be extended to $n=2$.  The polytope
$\Omega_1(d,2)$ is 1-dimensional and its two vertices are the
permutation matrices. When $d$ is odd, $\Zer_1(d,2)$ consists of the
permutation matrices, but any positive linear combination of them has
only positive entries and hence has positive permanent.  It follows
that $\Zer_1(d,2)$ is $0$-dimensional in this case.

\cjref{cj:centreface} implies that there are only finitely many local
extrema of the permanent on $\Omega_1(d,n)$. Every zero is a local
extremum, since the permanent is non-negative.  Thus \cyref{ZeroDim}
tells us that there are uncountably many local extrema.

\begin{cor}\label{cy:Tarcjfalse} 
\cjref{cj:centreface} is false for even $n$ and odd $d$.
\end{cor}

\cyref{ZeroDim} establishes that the dimension of $\Zer_1(d,n)$ is
greater than zero when $n>2$ is even and $d$ is odd.  It is likely that
$\Zer_1(d,n)$ is much larger than the bound that we have given, so we
have not worked hard to find the best constant $c$.  It is also worth
remarking that if \cjref{ZeroConj} holds then the zero set is empty
except when $n$ is even and $d$ is odd, which is the case we have 
concentrated on in this section.

\section{Counterexamples to the Dow-Gibson conjecture}\label{s:DG}

In this section we investigate counterexamples to \cjref{DGconj} in odd dimensions.
\tref{JMinMax} strongly suggests, but does not prove, that
\cjref{DGconj} should fail in odd dimensions. The issue is that
\tref{JMinMax} is a statement about 1-polystochastic matrices, whereas
\cjref{DGconj} deals with $(d-1)$-polystochastic matrices.  It is
plausible that $n^{-1}J^d_n$ may be the only matrix that is
1-polystochastic and which (up to scaling) is in the convex hull of
the $(d-1)$-permutation matrices.  As an aside, it certainly is in
that hull, because it is the average of all $(d-1)$-permutation
matrices, by symmetry.

If we can find any $1$-polystochastic matrix \emph{other than
  $n^{-1}J_n^d$} which (up to scaling) is in the convex hull of the
$(d-1)$-permutation matrices, then \tref{JMinMax} will imply that
\cjref{DGconj} fails in odd dimensions. That is what we do in this
section.

\begin{lemma}\label{DecompSquares} 
A permutation matrix $P\in \Lambda_1(3,n)$ has a decomposition $P=Q_1+\cdots+Q_n$ into matrices $Q_i\in\Lambda_2(3,n)$ if and only if $\corh(P)$ has an orthogonal mate.
\end{lemma}

\begin{proof}
A Latin square has an orthogonal mate if and only if it can be covered by a set of mutually disjoint transversals. Let $L=\corh(P)$. Recall that each transversal in $L$ corresponds to a positive diagonal in $P$, that is, a matrix $Q\in\Lambda_2(3,n)$ with $\Supp(Q)\subseteq\Supp(P)$, and vice versa. That transversals $T_1,\dots, T_n$ are mutually disjoint and cover $L$ is equivalent to the corresponding matrices $Q_1,\dots,Q_n$ summing to $P$, that is, $P=Q_1+\cdots+Q_n$.
\end{proof}

\begin{lemma}\label{DG3}  
\cjref{DGconj} is false for $d=3$ and all $n>2$.
\end{lemma}

\begin{proof} 
By \cite{MOLSpairs}, for $2<n\ne 6$ there exists a pair of orthogonal Latin squares $L_1$ and $L_2$ of order $n$, for which $\corp(L_1)$ and $\corp(L_2)$ decompose into elements of $\Lambda_2(3,n)$ by \lref{DecompSquares}. For $n=6$, we now give an example of a matrix $A_6\in\Omega_1(3,6)$ that is a convex combination of elements of $\Lambda_2(3,6)$:
\begin{align*} A_6&=\left[\large\begin{smallmatrix}1&0&0&0&0&0\\0&1&0&0&0&0\\0&0&1&0&0&0\\0&0&0&1&0&0\\0&0&0&0&1&0\\ 0&0&0&0&0&1\end{smallmatrix}\right| \large\begin{smallmatrix}0&0.5&0&0&0.5&0\\1&0&0&0&0&0\\0&0&0&1&0&0\\0&0&0&0&0&1\\0&0&1&0&0&0\\ 0&0.5&0&0&0.5&0\end{smallmatrix}\left| \begin{smallmatrix}0&0&1&0&0&0\\0&0&0&0&0&1\\1&0&0&0&0&0\\0&0&0&0&1&0\\0&1&0&0&0&0\\ 0&0&0&1&0&0\end{smallmatrix}\right| \begin{smallmatrix}0&0&0&1&0&0\\0&0&0&0&1&0\\0&1&0&0&0&0\\1&0&0&0&0&0\\0&0&0&0&0&1\\ 0&0&1&0&0&0\end{smallmatrix}\left| \begin{smallmatrix}0&0.5&0&0&0.5&0\\0&0&0&1&0&0\\0&0&0&0&0&1\\0&0&1&0&0&0\\1&0&0&0&0&0\\ 0&0.5&0&0&0.5&0\end{smallmatrix}\right|\left. \begin{smallmatrix}0&0&0&0&0&1\\0&0&1&0&0&0\\0&0&0&0&1&0\\0&1&0&0&0&0\\0&0&0&1&0&0\\ 1&0&0&0&0&0\end{smallmatrix}\right]\\
&= \left[E_{11}\right| E_{53}\left| E_{64}\right|E_{25}\left|E_{36}\right|\left. E_{42}\right]+ \left[E_{55}\right| E_{46}\left| E_{13}\right|E_{32}\left|E_{24}\right|\left. E_{61}\right]\\
&\quad+ \left[E_{66}\right| E_{21}\left| E_{52}\right|E_{14}\left|E_{43}\right|\left. E_{35}\right]+ \left[E_{22}\right| E_{34}\left| E_{45}\right|E_{63}\left|E_{51}\right|\left. E_{16}\right]\\
&\quad+\frac{1}{2} \left[E_{33}\right| E_{12}\left| E_{26}\right|E_{41}\left|E_{65}\right|\left. E_{54}\right]+\frac{1}{2} \left[E_{33}\right| E_{15}\left| E_{26}\right|E_{41}\left|E_{62}\right|\left. E_{54}\right]\\
&\quad+\frac{1}{2} \left[E_{44}\right| E_{65}\left| E_{31}\right|E_{56}\left|E_{12}\right|\left. E_{23}\right]+\frac{1}{2} \left[E_{44}\right| E_{62}\left| E_{31}\right|E_{56}\left|E_{15}\right|\left. E_{23}\right].
\end{align*}
Here $E_{ij}$ is the $6\times 6$ matrix with a single $1$ in cell $(i,j)$ and zeros elsewhere, and 3-dimensional matrices are specified by listing their layers  separated by vertical bars. 

For $n>2$, we have shown there exists $A_n\in\Omega_1(3,n)$ such that $n^{-1}A_n\in\Hull(\Lambda_2(3,n))$ and $A_n\ne n^{-1}J_n^3$. By Theorem \ref{JMinMax}, we know $\Per\left((1-\epsilon)n^{-1}J_n^3+\epsilon A_n\right)<\Per(n^{-1}J_n^3)$ for $\epsilon$ sufficiently small, and hence $\Per\left((1-\epsilon)n^{-2}J_n^3+\epsilon n^{-1}A_n\right)<\Per(n^{-2}J_n^3)$ for $\epsilon$ sufficiently small. 
\end{proof}

To address the rest of the odd dimensional case, we use the following notion of the product of two higher dimensional matrices. By this definition, the product of two matrices increases in dimension while preserving enough of the properties of being polystochastic or a permutation matrix to get the same conclusion as in Lemma \ref{DG3}.

For the following we write $a_{i_1\cdots i_d}$ as shorthand for $A(i_1,\dots,i_d)$.
\begin{Def} 
Let $A\in M(p,n)$ and $B\in M(q,n)$. Define $A\times B\in M(p+q-2,n)$ by
\[(A\times B)_{i_1\cdots i_{p+q-2}}=\sum_j a_{i_1\cdots i_{p-1}j}b_{ji_p\cdots i_{p+q-2}}.\]
\end{Def}

Viewing $A$ and $B$ as the arrays corresponding to tensors under some choice of basis of a vector space, this multiplication corresponds to a tensor contraction of the tensor product of $A$ and $B$.
Multiplication of higher dimensional matrices is associative and distributive.

\begin{lemma}\label{MultPoly} 
Suppose $A\in\Omega_1(p,n)$ and $B\in\Omega_1(q,n)$. Then $A\times B\in\Omega_1(p+q-2,n)$.
\end{lemma}

\begin{proof} 
First consider a line sum in $C=A\times B$ with free index $i_r\in \{i_1,\dots, i_{p-1}\}$. Then 
\begin{align*}\sum_{i_r}c_{i_1\cdots i_{p+q-2}}=\sum_{i_r,j}a_{i_1 \cdots i_{p-1}j}b_{ji_p\cdots i_{p+q-2}}=\sum_{j}b_{ji_p\cdots i_{p+q-2}}=1.
\end{align*} 
Similarly for a line with free index $i_r\in \{i_p,\dots, i_{p+q-2}\}$, 
\begin{align*}\sum_{i_r}c_{i_1\cdots i_{p+q-2}}&=\sum_{i_r,j}a_{i_1\cdots i_{p-1}j}b_{ji_p\cdots i_{p+q-2}}=\sum_ja_{i_1\cdots i_{p-1}j}=1.\qedhere
\end{align*}
\end{proof}

\begin{lemma} 
Suppose $A\in\Omega_s(p,n)$ and $B\in\Omega_t(q,n)$. Then $A\times B\in\alpha\Omega_r(p+q-2,n)$ for $r\ge \max\{p+t-2,q+s-2\}$ and $\alpha=n^{r-s-t+1}$.
\end{lemma}

\begin{proof} Consider an $r$-plane sum in $C=A\times B$ with free indices $\{i_{l_1},\dots,i_{l_k}\}\subseteq\{i_1,\dots,i_{p-1}\}$ and $\{i_{l_{k+1}},\dots,i_{l_r}\}\subseteq \{i_p,\dots, i_{p+q-2}\}$. Since $r\ge q+s-2$, we must have $k\ge s-1$ as at most $q-1$ of the free indices can be from $\{i_p,\dots, i_{p+q-2}\}$. If $k\ge s$, \begin{align*}\sum_{i_{l_1},\dots, i_{l_r}}c_{i_1\cdots i_{p+q-2}}&=\sum_{i_{l_1},\dots, i_{l_r},j}a_{i_1\cdots i_{p-1}j}b_{ji_p\cdots i_{p+q-2}}\\&=n^{k-s}\sum_{i_{l_{k+1}},\dots,i_{l_r},j}b_{ji_p\cdots i_{p+q-2}}=n^{k-s}n^{r-k+1-t}=n^{r-s-t+1},\end{align*} noting that, since $k\le p-1$, we have $r-k+1\ge r-(p-1)+1\ge t$. Otherwise $k=s-1$, so 
\begin{align*}
\sum_{i_{l_1},\dots, i_{l_r}}c_{i_1\cdots i_{p+q-2}}&=\sum_{i_{l_1},\dots, i_{l_r},j}a_{i_1\cdots i_{p-1}j}b_{ji_p\cdots i_{p+q-2}}\\
&=n^{r-s-t+1}\sum_{i_{l_1},\dots,i_{l_k},j}a_{i_1\cdots i_{p-1}j}=n^{r-s-t+1}.\qedhere
\end{align*}
\end{proof}

Since the product of two permutation matrices is clearly a $(0,1)$-matrix, we get the following corollaries.

\begin{cor} 
Suppose $A\in\Lambda_1(p,n)$ and $B\in\Lambda_1(q,n)$. Then $A\times B\in \Lambda_1(p+q-2,n)$.
\end{cor}

\begin{cor}\label{MultOnes} 
Suppose $A\in\Lambda_s(p,n)$ and $B\in\Lambda_t(q,n)$. Then $A\times B\in \alpha\Lambda_r(p+q-2,n)$ for $r\ge\max\{p+t-2,q+s-2\}$ and $\alpha=n^{r-s-t+1}$.
\end{cor}

We are now ready to generalise \lref{DG3} to higher odd dimensions.

\begin{thm} 
\cjref{DGconj} is false for all odd $d$ and all $n>2$.
\end{thm}

\begin{proof} We use induction on $d$ to show that for any $d\ge3$ and $n\ge3$ there is $A\in\Omega_1(d,n)$ with $n^{2-d}A\in\Hull(\Lambda_{d-1}(d,n))$.  From the proof of Lemma \ref{DG3}, such an $A$ exists for $d=3$ and any $n\ge3$. Now suppose we have $A\in\Omega_1(d,n)$ with $n^{2-d}A\in\Hull(\Lambda_{d-1}(d,n))$, with decomposition $n^{2-d}A=\sum_i c_iP_i$, where $P_i\in\Lambda_{d-1}(d,n)$ for all $i$ and $\sum_i c_i=1$. Let $B\in\Omega_1(3,n)\cap n\Hull(\Lambda_2(3,n))$. By \lref{MultPoly}, $A\times B\in\Omega_1(d+1,n)$. Also 
\[A\times B=\left(\sum_in^{d-2}c_iP_i\right)\times \left(\sum_j nd_jQ_j\right)=\sum_{i,j}n^{d-1}c_id_jP_i\times Q_j,\] 
where $Q_j\in\Lambda_2(3,n)$ for all $j$ and $\sum_j d_j=1$. By Corollary \ref{MultOnes}, $P_i\times Q_j\in\Lambda_d(d+1,n)$ for every $i,j$. Lastly, $\sum_{i,j} c_id_j=1$ and so $n^{1-d}A\times B\in\Hull(\Lambda_d(d+1,n))$ as required. The result then follows from \tref{JMinMax} and the same argument as in \lref{DG3}.
\end{proof}

We have showed that \cjref{DGconj} fails in odd dimensions. The even
dimensional case seems harder.  As mentioned in the introduction,
Taranenko \cite{TaranenkoBig} found counterexamples to
\cjref{DGconjmod} for order 3. We would like to be able to use her examples
to show that Conjecture \ref{DGconj} fails for order 3, and possibly
further. In order to do so, we need a way to decompose permutation
matrices in $\Lambda_1(d,n)$ into diagonals in
$\Lambda_{d-1}(d,n)$. Unfortunately, it seems difficult to find such
decompositions in general.  Jurkat and Ryser \cite{JurkatRyser}
consider some conditions that make these decompositions possible. In
particular they show that finding a decomposition is equivalent to
extending a set of mutually orthogonal Latin hypercubes.

\subsection*{Acknowledgement} 
The authors thank Daniel Horsley for motivating discussions and Anna Taranenko
for helpful comments on a draft of this paper.

\end{document}